\documentclass[12pt,reqno]{amsart}
\usepackage{amsmath,amssymb,amscd,amsthm,amscd,amsthm,amsfonts,mathrsfs,array,comment}
\usepackage{graphicx,color}
\usepackage[all]{xy}
\usepackage{tikz}
\usetikzlibrary{intersections,calc,arrows.meta}
\usepackage{xcolor}   
\usepackage[hidelinks]{hyperref}

\newcommand{\R}{{\mathbb{R}}}
\newcommand{\Z}{{\mathbb{Z}}}
\newcommand{\C}{{\mathbb{C}}}

\newcommand{\F}{{\mathbb{F}}}
\newcommand{\Fk}{{\F_k}}
\newcommand{\ii}{{\mathbf{i}}}
\newcommand{\ee}{{\mathbf{e}}}
\newcommand{\fm}{{\mathfrak{m}}}
\newcommand{\cO}{{\mathcal{O}}}
\newcommand{\cE}{{\mathcal{E}}}

\newcommand{\cI}{{\mathcal{I}}}

\newcommand{\cV}{{\mathcal{V}}}

\newcommand{\del}{{\partial}}
\newcommand{\rt}{{\sqrt{-1}}}
\newcommand{\pr}{{\mathrm{pr}}}
\newcommand{\grad}{{\mathrm{grad}}}
\newcommand{\Int}{{\mathrm{Int}}}

 \newtheorem{thm}{Theorem}[section]
 \newtheorem{lem}[thm]{Lemma}
 \newtheorem{cor}[thm]{Corollary}

\theoremstyle{definition}

 \newtheorem{rem}[thm]{Remark}

\begin{document}

\title{SYZ mirror of Hirzebruch surfaces $\mathbb{F}_k$ and Morse homotopy}
\author{HAYATO NAKANISHI}
\address{Department of Mathematics and Informatics, Graduate School of Science and Engineering, Chiba University, 1-33
Yayoicho, Inage, Chiba, 263-8522 Japan.}
\email{hayato\underline{ }nakanishi@chiba-u.jp}
\date{}

\maketitle

\begin{abstract}
We study homological mirror symmetry for Hirzebruch surfaces $\mathbb{F}_k$ as complex manifolds by using the Strominger-Yau-Zaslow construction of mirror pair and Morse homotopy. For  toric Fano surfaces, Futaki-Kajiura and the author proved homological mirror symmetry by using Morse homotopy in \cite{fut-kaj1,fut-kaj2,HN}. In this paper, we extend Futaki-Kajiura's result of the Hirzebruch surface $\mathbb{F}_1$ to $\mathbb{F}_k$. We discuss Morse homotopy and show that homological mirror symmetry in the sense above holds true.
\end{abstract} 

\tableofcontents

\section{Introduction.}
 Strominger-Yau-Zaslow \cite{SYZ} proposed a construction of a mirror pair using dual torus fibrations. Recently, various mirror pairs have been found to have dual torus fibrations. Kontsevich-Soibelman \cite{KoSo:torus} proposed a framework to systematically prove homological mirror symmetry via Morse homotopy for dual torus fibrations over a closed manifold without singular fibers.  Fukaya-Oh \cite{FO} proved that the category $Mo(B)$ of Morse homotopy on a compact manifold $B$ is equivalent to the Fukaya category $Fuk(T^*B)$ of the cotangent bundle $T^*B$, and Kontsevich-Soibelman's approach is based on this result. Futaki-Kajiura and the author apply SYZ construction to toric Fano surfaces as complex manifolds and discuss explicitly homological mirror symmetry.\par
Fano surfaces are also called del Pezzo surfaces, and homological mirror symmetry of toric del Pezzo surfaces is discussed by Auroux-Katszarkov-Orlov \cite{AKO} and Ueda \cite{U}. In these papers, they treat the toric del Pezzo surfaces as complex manifolds and consider a Fukaya-Seidel category \cite{seidel} corresponding to a Landau-Ginzburg superpotential of toric del Pezzo surfaces as the symplectic side. Also, Abouzaid showed homological mirror symmetry of toric varieties by using Morse homotopy on a polytope in \cite{A}.
The polytope used in his approach is different from that used in our approach.
 The polytope in \cite{A} is a connected component of the complement of the tropical amoeba for the general fiber of the Laurent polynomial. On the other hand, the polytope in our papers is just the moment polytope of toric manifolds itself.
 Also, the category of Morse homotopy, defined by Futaki-Kajiura \cite{fut-kaj1}, is slightly different from that defined by Abouzaid \cite{A}. A space of morphisms of the category defined by Abouzaid is spanned by non-degenerate critical points. On the other hand,  a space of morphisms of the category defined by Futaki-Kajiura is spanned by critical points that can be degenerated and be in the boundary of the moment polytope.
\par
In this paper, we extend the Futaki-Kajiura's result of the Hirzebruch surface $\F_1$ in \cite{fut-kaj2} to the non-Fano case $\Fk$.
Namely, we consider Hirzebruch surfaces $\Fk$ as complex manifolds and discuss homological mirror symmetry by using the SYZ construction and Morse homotopy. In the case of $\Fk$, we can not express the dual coordinates of the SYZ fibrations explicitly, unlike the case of $\F_1$.
 Therefore, we can not directly compute the intersections of Lagrangian sections with respect to the coordinate of the moment polytope. 
 However, by computing the intersection of Lagrangian sections with respect to the dual coordinates as in the previous paper \cite{HN}, we can investigate the intersection of Lagrangian sections. Furthermore, the Hirzebruch surface $\Fk$ have a full strongly exceptional collection. Therefore, it is sufficient to consider a full subcategory consisting of the full strongly exceptional collection.
 In view of this situation, we prove that Hirzebruch surfaces $\Fk$ satisfy homological mirror symmetry for any positive integer $k$. On the other hand, the category of Morse homotopy of $\Fk$ behaves differently for $k=1$ or $k\geq2$. We also investigate the difference between these categories.\par
This paper is organized as follows. In section 2, we recall the SYZ fibration set-up and explain a plan for proving homological mirror symmetry of Hirzebruch surfaces $\Fk$. In section 3, we review the explicit SYZ construction for $\Fk$ used in \cite{fut-kaj2} and extend the correspondence in $\F_1$ between Lagrangian sections and holomorphic line bundles to that in $\Fk$. 
\par
\noindent
{\bf Acknowledgments.}
I would like to thank my supervisor Professor Hiroshige Kajiura for various advices in writing this paper.
This work was supported by JST SPRING, Grant Number JPMJSP2109.

\section{Preliminaries.}
In this section, we briefly review the SYZ construction following \cite{LYZ,leung05,fut-kaj1,fut-kaj2,HN}.
\subsection{Torus bundle and dual torus bundle.}\label{MM}
Let $(\check{X},g_{\check{M}})$ be a smooth compact toric manifold where $g_{\check{M}}$ is a K\"ahler metric. For a moment map $\mu:\check{X}\to P$, we put $B=\Int(P)$ and $\check{M}:=\mu^{-1}(B)$. Then, the restriction of $\mu$ to $\check{M}$ forms a torus bundle. For a suitable coordinate, the metric on $\check{M}$ induced by that of $\check{X}$ expressed as
\begin{equation*}
g_{\check{M}} = \sum g^{ij}(dx_idx_j+dy_idy_j)
\end{equation*}
where $\check{x}=(x_1,\cdots,x_n)$ are the base coordinates and $\check{y}=(y_1,\cdots,y_n)$ are the fiber coordinates. By this expression, we obtain the metric $g_B$ on $B$ induced by that on $\check{M}$ and a diffeomorphism $TB\cong T^*B$. 
 Using the metric $g_B$ on $B$, we define the dual coordinates as follows: 
since $\sum_{j=1}^ng_{ij}dx^j$ is a closed form, there exists a function $x_i:=\phi_i$ of $x$ for each $i=1,\dots,n$ such that
\begin{equation*}
dx_i=\sum_{j=1}^ng_{ij}dx^j,
\end{equation*}
where $\{g_{ij}\}$ is the inverse matrix of $\{g^{ij}\}.$ Thus, we obtain the dual coordinates $x=(x_1,\dots,x_n)^t$ and the K\"ahler form $\omega_{TB}$ on $TB$ expressed by
\begin{equation*}
\omega_{TB}=\sum_{ij}g_{ij}dx^i\wedge dy^j. 
\end{equation*}
We assume that $B$ is a tropical affine manifold, i.e., the coordinate transformations are in $\R\rtimes GL_n(\Z)$. Then, by the fiberwise $\Z^n$-actions
\begin{equation*}
y\mapsto y+2\pi I,\ \ \ \check{y}\mapsto \check{y}+2\pi I
\end{equation*}
where $y$ is the fiber coordinate of $T^*B$ and $I\in\Z^n$, the quotients of $TB$ and $T^*B$ form  the torus bundles $TB/2\pi\Z^n$ and $T^*B/2\pi\Z^n$. We have $\check{M}\cong T^*B/2\pi\Z^n$ and define $M:=TB/2\pi\Z^n$. Moreover, the K\"ahler form $\omega_M$ on $M$ is induced by $\omega_{TB}$. Therefore, we obtain the two torus bundles (see Figure \ref{tfdtf}).
\begin{figure}[h]
\begin{equation*}
  \xymatrix{
   M \ar[rd]& & \check{M} \ar[ld] \\
   & B &  \\   
  } 
\end{equation*}
\caption{The torus fibration and the dual torus fibration}
\label{tfdtf}
\end{figure}

\subsection{Lagrangian sections of $M$ and Holomorphic line bundles over $\check{M}$.}\label{lag-line}
Let $\underline{s}:B\to M$ be a section of $M \to B$. Locally, we may regard $\underline{s}$ as a section of $TB \simeq T^*B$. Then, $\underline{s}$ is locally described by a collection of functions as 
\begin{equation*}
y^i = s^i(x).
\end{equation*}
\par
Now, under our assumption $TB\simeq T^*B$, we can check whether the graph of a section $\underline{s}:B\to M$ is Lagrangian or not in $T^*B$. The section $\underline{s}:B\to M$ is regarded as a section of $T^*B$ by setting $y_i=\sum_{j=1}^ng_{ij}y^j=\sum_{j=1}^ng_{ij}s^j$, from which one has
\begin{equation*}
\sum_{i=1}^n y_idx^i = \sum_{i=1}^n\left(\sum_{j=1}^ng_{ij}s^j\right)dx^i = \sum_{j=1}^ns^jdx_j.
\end{equation*}
Thus, the graph of the section $\underline{s}:B\to M$ is Lagrangian if and only if there exists a function $f$ such that $df=\sum_{j=1}^ns^jdx_j$ locally. The gradient vector field of such a function $f$ is of the form:
\begin{equation}\label{gra}
\grad(f) = \sum_{i,j}\frac{\del f}{\del x^j}g^{ji}\frac{\del}{\del x^i} = \sum_i\frac{\del f}{\del x_i}\frac{\del}{\del x^i}.
\end{equation}

Next, we define a line bundle $V$ with a $U(1)$-connection on the mirror manifold $\check{M}$ associated to $\underline{s}$. We set the covariant derivative locally as
\begin{equation*}
D := d - \frac{\ii}{2\pi}\sum_{i=1}^ns^i(x)dy_i.
\end{equation*}
By choosing suitable transition functions, the covariant derivative $D$ is defined globally. Moreover, the condition that $D$ defines a holomorphic line bundle on $\check{M}$ is equivalent to that the graph of $\underline{s}$ is Lagrangian in $M$. 

\subsection{DG categories $\cV$ and $DG(\check{X})$ of holomorphic line bundles.}\label{lineDG}
First, we define the DG category $\cV$ associated with $\check{M}$. The objects are holomorphic line bundles $V$ with $U(1)$-connections
\begin{equation}\label{connection}
D := d - \frac{\ii}{2\pi}\sum_{i=1}^ny^i(x)dy_i.
\end{equation}
These line bundles $(V,D)$ are associated to the Lagrangian section $\underline{y}:B\to M$. For two objects $(V_a,D_a),(V_b,D_b)\in\cV$, the space $\cV((V_a,D_a),(V_b,D_b))$ of morphisms is defined by
\begin{equation*}
\cV((V_a,D_a),(V_b,D_b)) := \Gamma(V_a,V_b)\otimes_{C^\infty(\check{M})}\Omega^{0,*}(\check{M}),
\end{equation*}
where $\Gamma(V_a,V_b)$ is the space of homomorphisms from $V_a$ to $V_b$ and $\Omega^{0,*}(\check{M})$ is the space of anti-holomorphic differential forms on $\check{M}$. The space $\cV((V_a,D_a),(V_b,D_b))$ is a $\Z$-graded vector space whose $\Z$-grading is defined by the degree of the anti-holomorphic differential forms. Let  $\cV^r((V_a,D_a),(V_b,D_b))$ denote the degree $r$ part. We define a linear map $d_{ab}:\cV^r((V_a,D_a),(V_b,D_b))\to\cV^{r+1}((V_a,D_a),(V_b.D_b))$ as follows. We decompose $D_a$ into its holomorphic part and anti-holomorphic part $D_a= D_a^{(1,0)} + D_a^{(0,1)}$, and set a $d_a:=2D_a^{(0,1)}$. Then, for $\psi_{ab}\in\cV^r((V_a,D_a),(V_b,D_b))$, we set
\begin{equation*}
d_{ab}(\psi_{ab}) := d_b\psi_{ab} -(-1)^r\psi_{ab}d_a.
\end{equation*}
This linear map $d_{ab}:\cV^r((V_a,D_a),(V_b.D_b))\to\cV^{r+1}((V_a,D_a),(V_b,D_b))$ satisfies $d_{ab}^2=0$ since $(V_a,D_a)$ and $(V_b,D_b)$ are holomorphic line bundles. The product structure $m:\cV((V_a,D_a),(V_b,D_b))\otimes\cV((V_b,D_b),(V_c,D_c))\to\cV((V_a,D_a),(V_{c},D_{c}))$ is defined by the composition of homomorphisms of line bundles together with the wedge product for the anti-holomorphic differential forms.\par
Next, we define the DG category $DG(\check{X})$ consisting of holomorphic line bundles on the toric manifold $\check{X}$. For a line bundle $V$ on $\check{X}$, we take a connection $D$, which defines a holomorphic structure on $V$. We set the objects of $DG(\check{X})$ as holomorphic line bundles $(V,D)$ on $\check{X}$ whose restriction to $\check{M}$ is isomorphic to a line bundle on $\check{M}$ with a connection of the form
\begin{equation*}
d-\frac{\ii}{2\pi}\sum_{i=1}^n y^i(x)dy_i.
\end{equation*}
The space $DG(\check{X})((V_a,D_a),(V_b,D_b))$ of morphisms is defined as the $\Z$-graded vector space whose degree $r$ part is given by
\begin{equation*}
DG^r(\check{X})((V_a,D_a),(V_b,D_b)) := \Gamma(V_a,V_b)\otimes_{C^\infty(\check{X})}\Omega^{0,r}(\check{X}),
\end{equation*}
where $\Gamma(V_a,V_b)$ is the space of smooth bundle maps from $V_a$ to $V_b$. The composition of morphisms and the DG structure is defined in a similar way as that in $\cV$ above. \par
We then have a faithful embedding $\cI:DG(\check{X})\to\cV$ by restricting the line bundles on $\check{X}$ to $\check{M}$. Denote by $\cV'$ the image $\cI(DG(\check{X}))$ of $DG(\check{X})$ under $\cI$.
The morphisms of the subcategory $\cV'$ have the boundary condition induced by the smoothness of the morphisms of the category $DG(\check{X})$ on the toric divisors.

\subsection{The category $Mo(P)$ of weighted Morse homotopy.}\label{Mo(P)}
Next, we consider the symplectic side.
We review the category $Mo(P)$ of Morse homotopy defined in \cite{fut-kaj1}.
\begin{itemize}
\item The objects : The objects are the Lagrangian sections $L$ of $\pi:M\to B$ which correspond to objects of $\cI(DG(\check{X}))\subset\cV$ described in subsection \ref{lineDG}\footnote{We can also interpret this condition as a condition of potential functions $f_L$ called the {\em growth condition} defined in \cite[Definition 3.1]{Chan}.}. We can extend $L$ to a section on $P$ smoothly. Note that there exists a function $f_L$ on $B$ such that $L$ is the graph of $df_L$ under the identification $TB\cong T^*B$ where the Lagrangian section $L$ is regarded as the gradient vector field $\sum\frac{\del f}{\del x_i}\frac{\del}{\del x^i}=\grad(f_L)\in\Gamma(TB)$ of $f_L$ as we see in (\ref{gra}).
\item The space of morphisms :  Given two objects $L,L'\in Mo(P)$, we assume that $L$ and $L'$ intersect cleanly. This means that there exists an open set $\tilde{B}$ such that $P\subset\tilde{B}$ and $L,L$ over $B$ can be extended to the graphs of smooth sections over $\tilde{B}$ so that they intersect cleanly.
 The space $Mo(P)(L,L')$ is the $\Z$-grading vector space spanned by the connected components $V$ of $\pi(L\cap L')\subset P$ which satisfy the following conditions:
\begin{description}
\item[(M1)] For each connected component $V\subseteq\pi(L\cap L')$, the dimension of the stable manifold $S_v\subset\tilde{B}$ of the gradient vector field $-\grad(f_L-f_{L'})$ at a generic point $v\in V$ is constant. Then we define the degree of $V$ by $|V|:=\dim(S_v)$
\item[(M2)] There exists a point $v\in V$ which is generic and is an interior point of $S_v\cap P\subset S_v$.
\end{description}

\item $A_\infty$-structure : In general, $Mo(P)$ is defined as an $A_\infty$-category, but the higher products $\fm_k\ (k\geq3)$ of the full subcategory $Mo_{\cE_c}(P)$ considered in subsection \ref{morse homotopy} vanish. 
This is why we only explain $\fm_1$ and $\fm_2$ here. For more details of the definition of the higher products $\fm_k$, see \cite{fut-kaj1}. 
Firstly, we take two Lagrangians $L,L'$ and connected components of the intersections $V,V'\subseteq \pi(L\cap L')$. Let $\mathcal{GT}(v;v')$ be the set of the gradient trees starting at $v\in V$ and ending at $v'\in V'$. We define $\mathcal{GT}(V;V') := \cup_{{v}\in V,v'\in V'}\mathcal{GT}(v;v')$ and denote $\mathcal{HGT}(V;V')$ its quotient by smooth homotopy. This set becomes a finite set when $|V'|=|V|+1$ and therefore we define the differential $\fm_1$ of morphisms by
\begin{equation*}
\begin{split}
\fm_1:&Mo(P)(L,L')\longrightarrow Mo(P)(L,L')\\
&V \longmapsto \sum_{\substack{V'\in Mo(P)(L,L')\\|V'|=|V|+1}}\sum_{[\gamma]\in \mathcal{HGT}(V;V')}e^{-A(\Gamma)}V',
\end{split}
\end{equation*}
where $A(\Gamma)\in[0,\infty]$ is the symplectic area of the piecewise smooth disc in $\pi^{-1}(\Gamma(T))$.
Secondly, take a triple $(L_1,L_2,L_3)$, connected components of the intersections $V_{12}\subseteq \pi(L_1\cap L_2), V_{23}\subseteq\pi(L_2\cap L_3), V_{13}\subseteq\pi(L_1\cap L_3)$. Let $\mathcal{GT}(v_{12},v_{23};v_{13})$ be the set of the trivalent gradient trees starting at $v_{12}\in V_{12}, v_{23}\in V_{23}$ and ending at $v_{13}\in V_{13}$. Define $\mathcal{GT}(V_{12},V_{23};V_{13}) := \cup_{v_{12}\in V_{12},v_{23}\in V_{23},v_{13}\in V_{13}}\mathcal{GT}(v_{12},v_{23};v_{13})$ and $\mathcal{HGT}(V_{12},V_{23};V_{13}):=\mathcal{GT}(V_{12},V_{23};V_{13})/\mathrm{smooth\ homotopy}$. This set become a finite set when $|V_{13}|=|V_{12}|+|V_{23}|$ and therefore we define the composition $\fm_2$ of morphisms by
\begin{equation*}
\begin{split}
\fm_2:&Mo(P)(L_1,L_2)\otimes Mo(P)(L_2,L_3)\longrightarrow Mo(P)(L_1,L_3)\\
&(V_{12},V_{23}) \longmapsto \sum_{\substack{V_{13}\in Mo(P)(L_1,L_3)\\|V_{13}|=|V_{12}|+|V_{23}|}}\sum_{[\gamma]\in \mathcal{HGT}(V_{12},V_{23};V_{13})}e^{-A(\Gamma)}V_{13},
\end{split}
\end{equation*}
where $A(\Gamma)\in[0,\infty]$ is the symplectic area of the piecewise smooth disc in $\pi^{-1}(\Gamma(T))$.
\end{itemize}

\subsection{Equivalence of DG-categories.}\label{plan}
For Hirzebruch surfaces $\Fk$, the derived category $D^b(Coh(\Fk))$ of coherent sheaves has a full strongly exceptional collection $\cE_c$ of line bundles\cite{hille-perling11,EL}. This means that $\cE_c$ generates $D^b(Coh(\Fk))$ in the sense that
\begin{equation*}
D^b(Coh(\Fk)) \simeq Tr(DG_{\cE_c}(\Fk)),
\end{equation*}
where $DG_{\cE_c}(\Fk)$ is the full DG subcategory of $DG(\Fk)$ consisting of $\cE_c$ and $Tr$ is the twisted complexes construction \cite{BK,kon94}. Also, we have a DG-quasi-isomorphism
\begin{equation*}
DG_{\cE_c}E(\Fk)\overset{\sim}\to\cV_{\cE_c}'=\cI(DG_{\cE_c}(\Fk)),
\end{equation*}
where $\cI:DG(\Fk)\to\cV$ is the faithful functor defined in subsection \ref{lineDG}. We denote the collection of the Lagrangian sections corresponding to the exceptional collection $\cE_c$ by the same symbol $\cE_c$. We denote by $Mo_{\cE_c}(P)$ the full subcategory of $Mo(P)$ consisting of $\cE_c$. Then the main result of this paper is the following:
\begin{thm}\label{main}
Let $\Fk$ be the Hirzebruch surface, $P$ be the moment polytope of $\Fk$, and $\cE_c$ be the full strongly exceptional collection of the derived category $D^b(Coh(\Fk))$ of coherent sheaves on $\Fk$.
Then we have a DG-quasi-isomorphisms
\begin{equation*}
Mo_{\cE_c}(P) \simeq \cV_{\cE_c}' \simeq DG_{\cE_c}(\Fk),
\end{equation*}
where $Mo_{\cE_c}(P)$ is the full subcategory of $Mo(P)$ consisting of the collection of the Lagrangian sections mirror to $\cE_c$.
\end{thm}

Futaki and Kajiura proved this theorem for $\F_1$ in \cite{fut-kaj2}. In this paper, we extend the result to any positive integer $k\in\Z_{\geq0}$.
\begin{rem}
If $k=2$ then the surface $\F_2$ is weak Fano, but if $k\geq3$ then the surface $\Fk$ is not even weak Fano. In other words, our results are examples of homological mirror symmetry of non-Fano cases.
\end{rem}
The next corollary immediately follows from the main theorem (Theorem \ref{main}) because the triangulated categories induced from DG-quasi-isomorphic categories are isomorphic as triangulated categories.

\begin{cor}\label{main cor}
For the Hirzebruch surfaces $\Fk$, we have an equivalence of triangulated categories
\begin{equation*}
Tr(Mo_{\cE_c}(P)) \simeq D^b(Coh(\Fk)),
\end{equation*}
where $P$ is the moment polytope of $\Fk$ and $\cE_c$ is the collection of the Lagrangian sections mirror to the full strongly exceptional collection of the holomorphic line bundles on $\Fk$.
\end{cor}

\section{Homological mirror symmetry of $\Fk$.}
In this section, we consider homological mirror symmetry for the Hirzebruch surface $\Fk$. We follow the convention of \cite{fut-kaj2}.
In subsection \ref{Fk}, we review the setting of the SYZ construction for $\Fk$ in \cite{fut-kaj2}. In subsection \ref{DGFk}, we discuss DG-categories consisting of holomorphic line bundles and consider a full strongly exceptional collection of $D^b(Coh(\Fk))$. In subsection \ref{morse homotopy}, we construct the Lagrangian sections which are SYZ mirror dual to the holomorphic line bundles in $DG(\Fk)$ and prove the main theorem (Theorem \ref{main}) for $\Fk$. In subsection \ref{minimality}, we give an example of the space of morphisms which is not minimal.

\subsection{Hirzebruch surface $\Fk$.}\label{Fk}
The Hirzebruch surface $\Fk$ is defined by
\begin{equation*}
\Fk :=\left\{ ([s_0:s_1],[t_0:t_1:t_2])\ \middle|\ (s_0)^kt_0=(s_1)^kt_1 \right\} \subset \C P^1\times\C P^2.
\end{equation*}
We take the following open covering $\{U_i\}_i$: 
\begin{align*}
U_1  =  \left\{([s_0:s_1],[t_0:t_1:t_2]) \in \Fk\ \middle|\ t_1 \neq 0\ ,\ s_0 \neq 0\right\},\\
U_2  =  \left\{([s_0:s_1],[t_0:t_1:t_2]) \in \Fk\ \middle|\ t_0 \neq 0\ ,\ s_1 \neq 0\right\},\\
U_3  =  \left\{([s_0:s_1],[t_0:t_1:t_2]) \in \Fk\ \middle|\ t_2 \neq 0\ ,\ s_0 \neq 0\right\},\\
U_4  =  \left\{([s_0:s_1],[t_0:t_1:t_2]) \in \Fk\ \middle|\ t_2 \neq 0\ ,\ s_1 \neq 0\right\}.
\end{align*}
The local coordinates in $U_2$ is $(u,v) := \left(\frac{s_0}{s_1},\frac{t_2}{t_0}\right)$. Hereafter, we fix $U:=U_2$.
There exists natural projections $\pr_1:\Fk\to\C P^1$ and $\pr_2:\Fk\to\C P^2$.
Using these projections, we define the K\"ahler form of $\Fk$ by
\begin{equation*}
\check{\omega}:= C_1\pr_1^*\left(\omega_{\C P^1}\right) + C_2\pr_2^*\left(\omega_{\C P^2}\right),
\end{equation*}
where $C_1 > 0$ and $C_2 > 0$ are real constants and $\omega_{\C P^n}$ is the Fubini-Study form on $\C P^n$. Correspondingly, the moment map $\mu:\Fk\to \R^2$ is given by
\begin{equation*}
\mu([s_0:s_1],[t_0:t_1:t_2]) :=\left(\frac{2C_1|s_1|^2}{|s_0|^2+|s_1|^2}+k\frac{2C_2|t_1|^2}{|t_0|^2+|t_1|^2+|t_2|^2}, \frac{2C_2|t_2|^2}{|t_0|^2+|t_1|^2+|t_2|^2}\right).
\end{equation*}
The image $\mu(\Fk)$ is called the moment polytope associated to $\check{\omega}$, which is the trapezoid. Namely, the moment polytope is
\begin{equation*}
P:=\left\{(x^1,x^2)\in\R^2\ \middle|\ 0\leq x^1 \leq 2(C_1+kC_2)-kx^2,\ 0\leq x^2 \leq 2C_2\right\}.
\end{equation*}
We denote each edge of $P$ by $E_i$  (see Figure \ref{poly1}).

\begin{figure}[h]
\center
\begin{tikzpicture}
\draw(0,0)--node[auto=left]{$E_2$}  (12,0)-- node[auto=left]{$E_3$} (4,4)--node[auto=left]{$E_4$}  (0,4)--node[auto=left]{$E_1$}  cycle;
\end{tikzpicture}
\caption{The moment polytope of $\Fk$.}
\label{poly1}
\end{figure}

Now, we set
\begin{equation*}
\check{M}:=U_0\cap U_1\cap U_2\cap U_3,\ \ \ B:=\mathrm{Int}P,
\end{equation*}
and we treat $\check{M}$ as a torus fibration $\mu|_{\check{M}}:\check{M}\to B$. Then, $\check{M}$ is equipped with an affine structure by $u=e^{x_1+\ii y_1}$ and $v=e^{x_2+\ii y_2}$, where $y_1$ and $y_2$ are the fiber coordinates of $\check{M}$. The K\"ahler form $\check{\omega}$ is expressed as
\begin{equation*}
\begin{split}
\check{\omega}  &= 4 C_2\frac{k^2(1+t)s^kdx_1\wedge dy_1 -ks^ktdx_1\wedge dy_2 - ks^ktdx_2\wedge dy_1 +(1+s^k)tdx_2\wedge dy_2}{(1+s^k+t)^2}\\
 &\ \ \ +4 C_1\frac{sdx_1\wedge dy_1}{(1+s)^2},
\end{split}
\end{equation*} 
 on $U$, where $s:=|u|^2 = e^{2x_1}$ and $t:=|v^2|= e^{2x_2}$. By this expression, the inverse matrix $\{g^{ij}\}$ of the metric $\{g_{ij}\}$ on $B$ is given by
\begin{equation*}
\begin{pmatrix}
g^{11} & g^{12}\\
g^{21} & g^{22}
\end{pmatrix}
=
\begin{pmatrix}
C_1\frac{4 s}{(1+s)^2} + C_2\frac{4k^2(1+t)s^k}{(1+s^k+t)^2} & C_2\frac{-4 ks^kt}{(1+s^k+t)^2} \\
C_2\frac{-4 ks^kt}{(1+s^k+t)^2} & C_2\frac{4(1+s^k)t}{(1+s^k+t)^2}
\end{pmatrix}.
\end{equation*}
Now, we put $\psi:=C_2\log(1+e^{2k x_1}+e^{2 x_2}) + C_1\log(1+e^{2 x_1})$ and then $\frac{\del^2 \psi}{\del x_i\del x_j}=g^{ij}$. Thus, the dual coordinate $(x^1,x^2)$ is obtained by
\begin{equation*}
\begin{split}
(x^1,x^2)
&:=
\left(\frac{\del\psi}{\del x_1} , \frac{\del\psi}{\del x_2}\right)\\
&=
\left(C_1\frac{2e^{2 x_1}}{1+e^{2 x_1}} + C_2k\frac{2e^{2k x_1}}{1+e^{2k x_1}+e^{2 x_2}}, C_2\frac{2e^{2 x_2}}{1+e^{2k x_1}+e^{2 x_2}}\right)\\
&=
\mu\left([e^{x_1+\ii y_1}:1],[1:e^{k(x_1+\ii y_1)}:e^{x_2+\ii y_2}]\right).
\end{split}
\end{equation*}
For simplicity, we fix $C_1=C_2=1$ since the structure of the category $Mo(P)$ we shall construct is independent of these constants.
Hereafter, we regard $M$ as the dual torus fibration of $\mu|_{\check{M}}:\check{M}\to B$.

\subsection{The category of complex side.}\label{DGFk}
Any line bundle over $\Fk$ is constructed from the toric divisors, which is a linear combination of the following divisors:
\begin{equation*}
 D_{12}=(t_2=0),\   D_{24}=(s_0=t_1=0),\   D_{13}=(s_1=t_0=0),\ 
 D_{34}=(t_0=t_1=0). 
\end{equation*}
Using these divisors, we can identify $\pi_1^*\cO_{\C P^1}(1)=\cO(D_{24})$ and $\pi_2^*\cO_{\C P^2}(1)=\cO(D_{12})$.
 Any line bundle over $\Fk$ is generated by $(D_{12},\ D_{24})$, so that we put 
 \begin{equation*}
 \cO(a,b):=\cO(aD_{24}+bD_{12}).
 \end{equation*}
A connection one-form of $\cO(a,b)$ is expressed as
 \begin{equation*}
A_{(a,b)}:=-a\frac{s(dx_1+\rt dy_1)}{1+s}  -b\frac{ks(dx_1+\rt dy_1)+t(dx_2+\rt dy_2)}{1+s^k+t}.
\end{equation*}
We consider $DG(\Fk)$ as defined in subsection 2.4, where the objects are the line bundles $\cO(a,b)$ with the connection one-form $A_{(a,b)}$. The DG structure of $DG(\Fk)$ is given by the way described in subsection \ref{lineDG}. Since each $\cO(a,b)$ is a line bundle, we have
\begin{equation*}
DG(\Fk)\left(\cO(a_1,b_1),\cO(a_2,b_2)\right)\simeq DG(\Fk)\left(\cO,\cO(a_2-a_1,b_2-b_1)\right).
\end{equation*}
In particular,  the zero-th cohomology of $DG(\Fk)(\cO(0,0),\cO(a,b))$ is the space $\Gamma(\Fk,\cO(a,b))$ of holomorphic global sections. 
\begin{equation*}
\begin{split}
H^0(DG(\Fk)(\cO(a_1,b_1),\cO(a_2,b_2))) &\simeq H^0(DG(\Fk)(\cO(0,0),\cO(a_2-a_1,b_2-b_1))\\
&\simeq \Gamma(\Fk,\cO(a_2-a_1,b_2-b_1)) .
\end{split}
\end{equation*}
Using the coordinates $(u,v)$ for $U$, the generators of $\Gamma(\Fk,\cO(a,b))$ are expressed explicitly as
\begin{equation*}
\psi_{(i_1,i_2)}:=u^{i_1}v^{i_2},
\end{equation*}
where $0\leq i_2\leq b$ and $0\leq i_1\leq a+k(b-i_2)$. For more details, see \cite{fulton93toric,coxtor,fut-kaj2}. \par
Next, we consider a full strongly exceptional collection of the derived category of coherent sheaves $D^b(Coh(\Fk))$. It is known (Hille-Perling \cite{hille-perling11} and Elagin-Lunts \cite{EL}) that
\begin{equation*}
\cE_c:=(\cO,\cO(1,0),\cO(c,1),\cO(c+1,1)).
\end{equation*}
with a fixed $c=0,1,2,\cdots$ forms a full strongly exceptional collection of $D^b(Coh(\Fk))$. We denote by $DG_{\cE_c}(\Fk)$ be the full subcategory of $DG(\Fk)$ consisting of $\cE_c$.

Let $\cV$ be a DG category of holomorphic line bundles over $\check{M}$. For $\cO(a,b)$, twisting the fibers of $\cO(a,b)$ by the isomorphisms $\Psi_{(a,b)}:=(1+s)^{\frac{a}{2}}(1+s^k+t)^{\frac{b}{2}}$, we can remove the $dx$ term from $A_{(a,b)}$ as follows:
\begin{equation*}
\Psi_{(a,b)}^{-1}(d+A_{(a,b)})\Psi_{(a,b)} = -a\frac{s\rt dy_1}{1+s}  -b\frac{ks\rt dy_1+t\rt dy_2}{1+s^k+t}
\end{equation*}
where $s=e^{2x_1}$ and $t=e^{2x_2}$.  Let $\widetilde{\cO}(a,b)$ denote the line bundle $\cO(a,b)$ twisted by $\Psi_{(a,b)}$. Then, the faithful functor $\cI:DG(\Fk)\to\cV$ assigns $\cO(a,b)$ to $\widetilde{\cO}(a,b)$. Also, each generator $\psi_{(i_1,i_2)}\in DG(\Fk)(\cO,\cO(a,b))$ is sent to be $\Psi^{-1}_{(a,b)}\psi_{(i_1,i_2)}\in\cV(\widetilde{\cO},\widetilde{\cO}(a,b))$, i.e.,
\begin{equation}\label{pre-e}
\begin{split}
\Psi^{-1}_{(a,b)}\psi_{(i_1,i_2)} &= (1+s)^{\frac{-a}{2}}(1+s^k+t)^{\frac{-b}{2}}u^{i_1}v^{i_2}\\
 &=(1+s)^\frac{-a}{2}(1 + s^k + t)^\frac{-b}{2}s^\frac{i_1}{2}t^\frac{i_2}{2}e^{\ii(i_1y_1 + i_2y_2)}.
\end{split}
\end{equation}
The images are denoted by $\cV':=\cI(DG(\Fk))$ and $\cV_\cE':=\cI(DG_\cE(\Fk))$, respectively. Then, the basis of $H^0(\cV(\widetilde{\cO},\widetilde{\cO}(a,b,c))$ are given by (\ref{pre-e}). If the functions $\Psi^{-1}_{(a,b)}\psi_{(i_1,i_2)}$ on $B$ extend to that on $P$ smoothly, then we rescale each basis $\Psi^{-1}_{(a,b)}\psi_{(i_1,i_2)}$ by multiplying a positive number and denote it by $\ee_{(a,b);(i_1,i_2)}$ so that
\begin{equation*}
\max_{x\in P}|\ee_{(a,b);(i_1,i_2)}(x)|=1.
\end{equation*}

\subsection{The category of symplectic side.}\label{morse homotopy}
We discuss the Lagrangian section $L(a,b)$ of the dual torus fibration $M\to B$ corresponding to the line bundle $\widetilde{\cO}(a,b)$. As we saw in subsection \ref{lag-line}, the Lagrangian section $L(a,b)$ corresponding to $\widetilde{\cO}(a,b)$ is obtained as the coefficients of $dy$ terms
\begin{equation}
\begin{pmatrix}
y^1 \\ y^2
\end{pmatrix}
=2\pi
\begin{pmatrix}
 a\frac{s}{1+s} +  bk\frac{s^k}{1+s^k+t} \\  b\frac{t}{1+s^k+t}
\end{pmatrix}.
\end{equation}
By this expression, we can extend $L(a,b)$ on $B$ to that on $P$ smoothly.
The potential function of this Lagrangian section is given by
\begin{equation}\label{fabc}
f =  \pi a\log(1+s) +  \pi b\log(1+s^k+t).
\end{equation}
The collection of the Lagrangian sections corresponding to the full strongly exceptional collection $\cE_c$ given in subsection \ref{DGFk} is denoted by the same symbol $\cE_c$, i.e.,
\begin{equation*}
\cE_c=\left(L(0,0),L(1,0),L(c,1),L(c+1,1)\right).
\end{equation*}\par
For the moment polytope $P$ of $\Fk$, we construct the full subcategory $Mo_{\cE_c}(P)\subset Mo(P)$ consisting of $\cE_c$, where $Mo(P)$ is defined in subsection \ref{Mo(P)}. The objects of $Mo(P)$ are Lagrangian sections $L(a,b)$ obtained above.  Since we have
\begin{equation*}
Mo(P)(L(a_1,b_1),L(a_2,b_2)) \simeq Mo(P)(L(0,0),L(a_2-a_1,b_2-b_1)),
\end{equation*}
we concentrate to computing the space $Mo_\cE(P)(L(0,0),L(a,b))$.
The intersections of $L(0,0)$ with $L(a,b)$ are expressed as
\begin{equation}\label{eq}
\begin{split}
2\pi
\begin{pmatrix}
i_1 \\ i_2
\end{pmatrix}
=
2\pi
\begin{pmatrix}
 a\frac{s}{1+s} +  bk\frac{s^k}{1+s^k+t} \\  b\frac{t}{1+s^k+t}
\end{pmatrix}
\end{split}
\end{equation} 
in the covering space of $\pi:\bar{M}\to P$, where $(i_1,i_2)\in\Z^2$, $s=e^{2x_1}$, $t=e^{2x_2}$, and $\bar{M}$ is a torus fibration formed by adding formally torus fibers on $\del P$ to the torus fibration $M\to B=\Int(P)$. If there exists a nonempty intersection, then we set $V_{(a,b);(i_1,i_2)}:=\pi(L(0,0)\cap L(a,b))$. We check  that $V_{(a,b);(i_1,i_2)}$ satisfies Conditions (M1) and (M2) given in subsection \ref{Mo(P)}. In this case, the gradient vector field associated to $V_{(a,b);(i_1,i_2)}$ is of the form
\begin{equation}\label{grad2}
2\pi\left(a\frac{s}{1+s} + bk\frac{s^k}{1+s^k+t} - i_1\right)\frac{\del}{\del x^1} + 2\pi\left(b\frac{t}{1+s^k+t} - i_2 \right)\frac{\del}{\del x^2}.
\end{equation}\par
On the other hand, for the space of the opposite directional morphisms, we have
\begin{equation*}
Mo(P)(L(a,b),L(0,0)) \cong Mo(P)(L(0,0),L(-a,-b)).
\end{equation*}
Thus, the connected component $V_{(a,b);I}$ coincides with the connected component $V_{(-a,-b);-I}$. The gradient vector field associated with $V_{(-a,-b);(-i_1,-i_2)}$ is the opposite direction of (\ref{grad2}).\par
\begin{figure}[h]
\center
\begin{tikzpicture}
\draw(0,0)--node[auto=left]{$E_2$} node[auto=right]{$t=0$} (12,0)-- node[auto=left]{$E_3$} (4,4)--node[auto=left]{$E_4$} node[auto=right]{$s\leq\infty,t=\infty$} (0,4)--node[auto=left]{$E_1$} node[auto=right]{$s=0$} cycle;
\draw(12,0)node[right]{$\frac{t}{s^k}=0$};
\draw(4,4)node[above right]{$\frac{t}{s^k}=\infty$};
\end{tikzpicture}
\caption{The moment polytope of $\Fk$.}
\label{poly2}
\end{figure}

The full strongly exceptional collections $\cE_c$ behave slightly differently when $c>0$ or $c=0$. We first consider the case of $c>0$. We will discuss the case of $c=0$ later.\par
We first discuss that when there exists a nonempty intersection of $L(a,b)$ with $L(0,0)$ in the covering space $\bar{M}\to P$. For $a,b\geq0$, since $0\leq s,t\leq\infty$, we have
\begin{equation*}
0\leq i_2=b\frac{t}{1+s^k+t}\leq b,\ \ \ 0\leq i_1+ki_2 = a\frac{s}{1+s} + bk\frac{s^k+t}{1+s^k+t}\leq a+kb
\end{equation*}
By (\ref{eq}), $t$ expressed as
\begin{equation*}
t = \frac{(1+s^k)i_2}{b-i_2}
\end{equation*}
and the equation (\ref{eq}) expressed as
\begin{equation*}
i_1 = a\frac{s}{1+s} + k(b-i_2)\frac{s^k}{1+s^k}.
\end{equation*}
If $0\leq i_1\leq a+k(b-i_2)$ there exists a solution of this equation since a function $f(x)=a\frac{x}{1+x} + k(b-i_2)\frac{x^k}{1+x^k}$ is monotonically increasing. By solving $y^j=2\pi i_j,\ j=1,2,$ we obtain the following.

\begin{lem}
We assume $a,b\geq0$ and $(a,b)\neq(0,0)$. For any $(i_1,i_2)$ satisfying
\begin{equation*}
0\leq i_2\leq b,\ \ \ 0\leq i_1\leq a+k(b-i_2),
\end{equation*}
the intersection $V_{(a,b);I}$ is nonempty and connected.
\begin{itemize}
\item If $a+k(b-i_2)=0$, then the intersection $V_{(a,b);I}$ is
\begin{equation*}
V_{(0,b);(0,b)} =\left\{ \left( x^1 , 2 \right)\in P\ \middle|\ 0\leq x^1\leq 2\right\}.
\end{equation*}
Note that the condition $a+k(b-i_2)=0$ is satisfied only when $a=0$ and $i_2=b$, under the above assumption.

\item If $b=0$, then the intersection $V_{(a,0);I}$ is
\begin{equation*}
V_{(a,0);(i_1,0)} = \left\{(x^1,x^2)\in P\ \middle|\  x^1+k\frac{i_1^k}{i_1^k+(a-i_1)^k}x^2=\frac{2i_1}{a}+k\frac{2i_1^k}{i_1^k+(a-i_1)^k}\right\}.
\end{equation*}

\item If $b\neq0$ and $a+k(b-i_2)\neq0$, then the intersection $V_{(a,b);I}$ consists of the point such that
\begin{equation*}
V_{(a,b);(i_1,i_2)} = \left\{\left(\frac{2s_{(i_1,i_2)}}{1+s_{(i_1,i_2)}}+\frac{2k(b-i_2)s_{(i_1,i_2)}^k}{b(1+s_{(i_1,i_2)}^k)},\frac{2i_2}{b}\right)\right\},
\end{equation*}
where $s_{(i_1,i_2)}$ satisfies $i_1=a\frac{s_{(i_1,i_2)}}{1+s_{(i_1,i_2)}}+k(b-i_2)\frac{s_{(i_1,i_2)}^k}{1+s_{(i_1,i_2)}^k}$.
\end{itemize}
\end{lem}

\begin{lem}
We assume $a,b\geq0$ and $(a,b)\neq(0,0)$. For any $(i_1,i_2)$ satisfying
\begin{equation*}
0\leq i_2\leq b,\ \ \ 0\leq i_1\leq a+k(b-i_2),
\end{equation*}
the intersection $V_{(a,b);I}$ forms a generator of $Mo(P)(L(0,0),L(a,b))$ of degree zero. Also, the intersection $V_{(a,b);I}$ does not form a generator of $Mo(P)(L(a,b),L(0,0))$.
\end{lem}
\begin{proof}
If $a+k(b-i_2)=0$ then the $a=0$, $i_2=b$, and the gradient vector field associated to $V_{(0,b);I}$ is of the form
\begin{equation*}
2\pi\left(bk\frac{s^k}{1+s^k+t}\right)\frac{\del}{\del x^1} + 2\pi\left(b\frac{t}{1+s^k+t} - b \right)\frac{\del}{\del x^2}.
\end{equation*}
For a point $v\in V_{(0,b);(0,b)}$, the stable manifold $S_v$ of the gradient vector field is $\{v\}$ itself since $b\frac{t}{1+s^k+t} - b\leq0$. So the intersection $V_{(0,b);(0,b)}$ is a generator of degree zero. If $b=0$ then the gradient vector field associated to $V_{(a,0);I}$ is of the form 
\begin{equation*}
2\pi\left(a\frac{s}{1+s} - i_1\right)\frac{\del}{\del x^1}.
\end{equation*}
For a point $v\in V_{(a,0);(i_1,0)}$, the stable manifold $S_v$ of the gradient vector field is $\{v\}$ itself, so the intersection $V_{(a,0);(i_1,0)}$ is a generator of degree zero. If $b\neq0$ and $a+k(b-i_2)\neq0$, then $V_{(a,b);(i_1,i_2)}$ consists of the point $v$ and the stable manifold $S_v$ of the gradient vector field is $\{v\}$ itself, so the intersection $V_{(a,b);(i_1,i_2)}$ is a generator of degree zero. \par
For $Mo(P)(L(a,b),L(0,0))\cong Mo(P)(L(0,0),L(-a,-b))$, the gradient vector field associated with $V_{(-a,-b);I}$ is the opposite direction to that of $V_{(a,b);I}$. For any $v\in V_{(-a,-b);I}$, the stable manifold $S_v$ of the gradient vector field is positive dimension and $V_{(-a,-b);I}$ is included $\del P$. Therefore $V_{(-a,-b);I}$ does not satisfy Condition (M2) and form a generator.
\end{proof}

\begin{lem}
For any $I\in\Z^2$ the intersection $V_{(-1,0);I}$ does not form a generator of $Mo(P)(L(0,0),L(-1,0))$. Also, the intersection $V_{(-d,-1);I}$ does not form a generator of $Mo(P)(L(0,0),L(-d,-1))$ if $d\geq0$.
\end{lem}
\begin{proof}
The intersection $V_{(-1,0);(0,0)}$ is $E_1$, and the intersection $V_{(-1,0);(-1,0)}$ is $E_3$.  The gradient vector field associated to $V_{(-1,0);I}$ is of the form
\begin{equation*}
2\pi\left(-\frac{s}{1+s} - i_1\right)\frac{\del}{\del x^1}.
\end{equation*}
For any $v\in V_{(-1,0);I}$, the stable manifold $S_v$ of the gradient vector field is dimension one. However, the intersections $V_{(-1,0);I}$ do not satisfy Condition (M2) in subsection \ref{Mo(P)} and can not be the generators. 
If $d>0$, then $V_{(-d,-1);I}$ consists of the point $v$ and the stable manifold $S_v$ of the gradient vector field is dimension two. Since the intersections $V_{(-d,-1);I}$ is included in $\del P$, the intersections $V_{(-d,-1);I}$ do not satisfy Condition (M2) and can not be the generators. 
If $d=0$ and $i_2\neq b$, then the intersection $V_{(0,b);I}$ is
\begin{equation*}
V_{(0,b);(i_1,i_2)} =\left\{ \left( \frac{2i_1^{\frac{1}{k}}}{(i_1)^{\frac{1}{k}}+(k(b-i_2)-i_1)^{\frac{1}{k}}} + \frac{2i_2}{b} , \frac{2i_2}{b} \right)\right\}.
\end{equation*} 
If $b=\pm1$, then $V_{(0,b);I}$ is included in $\del P$ and the stable manifold of the gradient vector field is dimension two. Therefore, the intersections $V_{(0,-1);I}$ do not satisfy Condition (M2) and can not be the generators. 
If $d=0$ and $i_2= b$, then the intersection $V_{(0,b);(0,b)}$ is $E_4$. For any $v\in V_{(0,-1);(0,-1)}$, the stable manifold of the gradient vector field is dimension one. However, the intersection $V_{(0,-1);(0,-1)}$ do not satisfy Condition (M2) in subsection \ref{Mo(P)} and can not be the generator. 
\end{proof}

If $c=0$, we need the following two lemmas.
\begin{lem}
For $i_2=0$ and $0\leq i_1\leq k-1$, the intersection $V_{(-1,1);I}$ forms a generator of $Mo(P)(L(0,0),L(-1,1))$ of degree zero. However, for $i_2=1$, the intersection $V_{(-1,1);I}$ does not form a generator.
\end{lem}
\begin{proof}
The gradient vector field associated to $V_{(0,b);I}$ is of the form
\begin{equation*}
2\pi\left(-\frac{s}{1+s}+k\frac{s^k}{1+s^k+t}-i_1\right)\frac{\del}{\del x^1} + 2\pi\left(b\frac{t}{1+s^k+t} - i_2 \right)\frac{\del}{\del x^2}.
\end{equation*}
If $0<i_1<k-1$ then the intersection $V_{(-1,1);(i_1,0)}$ is a point
\begin{equation*}
V_{(-1,1);(i_1,0)} = \left( \frac{2s_{i_1}}{1+s_{i_1}} + k\frac{2s_{i_1}^k}{1+s_{i_1}^k} , 0 \right),
\end{equation*} 
where $s_{i_1}$ satisfies $i_1 = -\frac{s_{i_1}}{1+s_{i_1}} + k\frac{s_{i_1}^k}{1+s_{i_1}^k}$.
If $i_1=0$ and $i_1=k-1$, then the intersection $V_{(-1,1);(i_1,0)}$ is not connected in both cases since the function $f(x)=-\frac{x}{1+x}+k\frac{x^k}{1+x^k}$ is not monotonically increasing. In fact, we have
\begin{equation*}
\begin{split}
V_{(-1,1);(0,0)} &= \{(0,0)\}\cup\left\{\left(-\frac{2s_{0}}{1+s_{0}} + k\frac{s_{0}^k}{1+s_{0}^k},0\right)\right\},\\
V_{(-1,1);(k-1,0)} &= \{(2(k+1),0)\}\cup\left\{\left(-\frac{2s_{k-1}}{1+s_{k-1}} + k\frac{s_{k-1}^k}{1+s_{k-1}^k},0\right)\right\},
\end{split}
\end{equation*}
where $s_{i_1}$ satisfies $i_1 = -\frac{s_{i_1}}{1+s_{i_1}} + k\frac{s_{i_1}^k}{1+s_{i_1}^k}$. For $\{(0,0)\}$ and $\{(2(k+1),0)\}$, the stable manifold $S$ of the gradient vector field is $S=\{(x^1,2)\}$ in both cases. These points are vertices of moment polytope, so they do not satisfy Condition (M2) and can not be the generators. For the others, the stable manifold is itself, so they form the generators of degree zero. We denote these generators by the same symbol $V_{(-1,1);I}$.
For $i_2=1$, the intersections $V_{(-1,1);I}$ are
\begin{equation*}
V_{(-1,1);(0,1)} = \left( 0 , 2 \right),\ \ \ V_{(-1,1);(-1,1)} = \left( 2 , 2 \right).
\end{equation*}
We see that the stable manifold of the gradient vector field is
\begin{equation*}
S_{V_{(-1,1);(i_1,1)}}=\{(x^1,2)\},
\end{equation*}
in both case $i_1=0$ and $i_1=-1$. However, since $V_{(-1,1);(0,1)}$ and $V_{(-1,1);(-1,1)}$ are vertices of the moment polytope $P$, so they can not be the generators.
\end{proof}

\begin{lem}
For any $I\in\Z^2$ the intersection $V_{(1,-1);I}$ does not form a generator of $Mo(P)(L(0,0),L(1,-1))$.
\end{lem}
\begin{proof}
The intersection $V_{(1,-1);I}$ is included in $\del P$ and the stable manifold $S_{V_{(1,-1);I}}$ is dimension one or two. Therefore, the intersections $V_{(1,-1);I}$ do not satisfy Condition (M2) and can not be the generators.
\end{proof}

By the above lemmas, we obtain the following.

\begin{lem}
The basis $\ee_{(a,b);I}$ of $H^0(\cV'_{\cE_c}(\widetilde{\cO}(a_1,b_1),\widetilde{\cO}(a_1+a,b_1+b)))$ are expressed as the form
\begin{equation*}
\ee_{(a,b);I}(x)=e^{-f_I}e^{\ii Iy},
\end{equation*}
where $e^{-f_I}$ is continuous on $P$ and smooth on $B$. Furthermore, the function $f_I$ satisfies
\begin{equation}\label{dfI}
df_I = \sum_{j=1}^2\frac{\del f_I}{\del x_j}dx_j,\ \ \ \frac{\del f_I}{\del x_j} = \frac{y^j_{(a,b)} - 2\pi i_j}{2\pi}
\end{equation}
in $B$ and $\min_{x\in P}f_I=0$. In particular, we have 
\begin{equation}\label{VI2}
\{x\in P\ |\ f_I(x)=0\}=V_{(a,b);I}.
\end{equation}
Thus, the correspondence $\iota:V_{(a,b);I}\mapsto \ee_{(a,b);I}$ gives a quasi-isomorphism
\begin{equation*}
\iota : Mo_{\cE_c}(P)(L(a_1,b_1),L(a_1+a,b_1+b))\to \cV'_{\cE_c}(\widetilde{\cO}(a_1,b_1),\widetilde{\cO}(a_1+a,b_1+b))
\end{equation*}
of complexes.
\end{lem}
\begin{proof}
By the expression (\ref{pre-e}), we see that the basis $\ee_{(a,b);I}$ of $H^0(\cV'_\cE(\widetilde{\cO}(a_1,b_1),\widetilde{\cO}(a_1+a,b_1+b)))$ are expressed as the form
\begin{equation*}
\ee_{(a,b);I} =c_{(a,b);(i_1,i_2)}(1+s)^\frac{-a}{2}(1 + s^k + t)^\frac{-b}{2}s^\frac{i_1}{2}t^\frac{i_2}{2}e^{\ii(i_1y_1 + i_2y_2)}
\end{equation*}
where $s = e^{2 x_1},\ t= e^{2 x_2}$ and $c_{(a,b);(i_1,i_2)}$ is a constant. Then, the functions $e^{-f_I}$ and $f_I$ are given by
\begin{equation}\label{fI2}
\begin{split}
e^{-f_I}& = c_{(a,b);I}(1+s)^{-\frac{a}{2}}(1 + s^k + t)^{-\frac{b}{2}}s^\frac{i_1}{2}t^\frac{i_2}{2},\\
f_I  &= \log\left((1+s)^\frac{a}{2}(1 + s^k + t)^\frac{b}{2}s^\frac{-i_1}{2}t^\frac{-i_2}{2}\right) + \mbox{const}  \\
&= \frac{a}{2}\log(1+e^{2 x_1}) + \frac{b}{2}\log(1+e^{2k x_1}+e^{2 x_2}) -  i_1 x_1 -  i_2 x_2 + \mbox{const}.
\end{split}
\end{equation}
By this expression, the function $f_I$ satisfies (\ref{dfI}).
If $c>0$, then we have $0 \leq i_1 \leq a+kb$ and $0 \leq i_2 \leq b$.
If $c=0$, then we have $i_2=0$ and $0 \leq i_1 \leq k-1$.
Therefore, $e^{-f_I}$ is continuous on $P$ and smooth on $B$. Furthermore, the defining equations of $V_{(a,b);I}$ coincide with the conditions which $x\in P$ is the minimum point of $f_I$. Since the DG structure of $Mo_{\cE_c}(P)$ is minimal, we obtain the last statement.
\end{proof}

For a full strongly exceptional collection $\cE_c$, we see that the space
\begin{equation*}
Mo(P)(L(a_1,b_1),L(a_1+a,b_1+b)) \simeq Mo(P)(L(0,0),L(a,b)),
\end{equation*}
satisfies $0\leq b\leq 1$. We call an element of $ Mo(P)(L(0,0),L(a,b))$ a morphism {\bf of type} $b$. Before we discuss the composition structure in $ Mo_{\cE_c}(P)$, we consider the gradient vector field associated to $V_{(a,b);I}\in Mo_{\cE_c}(P)(L(0,0),L(a,b))$ with $b=0$ and $b=1$.
If $b=0$, then the gradient vector field is of the form
\begin{equation*}
2\pi\left(a\frac{s}{1+s}-i_1\right)\frac{\del}{\del x^1}.
\end{equation*}
Thus, the gradient trajectories starting from $v=(v^1,v^2)\in V_{(a,b);I}$ are always in the line $x^2=v^2$.
If $b=1$, then the gradient vector field is of the form
\begin{equation*}
2\pi\left(a\frac{s}{1+s} + k\frac{s^k}{1+s^k+t} - i_1\right)\frac{\del}{\del x^1} + 2\pi\left(\frac{t}{1+s^k+t} - i_2 \right)\frac{\del}{\del x^2}.
\end{equation*}
In this case, we see that $V_{(a,1);(i_1,0)}$ belongs to $E_2\subset\del P$, and $V_{(a,1);(i_1,1)}$ belongs to $E_4\subset\del P$. This implies that the gradient trajectories starting from $v\in V_{(a,1);(i_1,0)}$ are always in the line $x^2 = 0$, and the gradient trajectories starting from $v \in V_{(a,1);(i_1,1)}$ are always in the line $x^2 = 2$.

\begin{thm}
The correspondence $\iota : Mo_{\cE_c}(P)\to \cV'_{\cE_c}$ is compatible with the composition. 
\end{thm}
\begin{proof}
The composition in $\cV_{\cE_c}'$ or $H^0(\cV_{\cE_c}')$ is simply a product between functions, i.e.,
\begin{equation}\label{prod-e-pre}
e^{-f_{(\alpha,\beta);I}}e^{\ii Iy} \otimes e^{-f_{(\alpha',\beta');J}}e^{\ii Jy} \longmapsto e^{-\left(f_{(\alpha,\beta);I}+f_{(\alpha',\beta');J}\right)}e^{\ii (I+J)y},
\end{equation}
where $\alpha:=a_2-a_1,\ \beta:=b_2-b_1,\ \alpha':=a_3-a_2,\ \beta':=b_3-b_2$.
Here, let $v$ be a point where the function $f_{(\alpha,\beta);I}+f_{(\alpha',\beta')}$ is minimum. Then, we can rewrite (\ref{prod-e-pre}) by \begin{equation}\label{prod-e}
\ee_{(\alpha,\beta);I} \otimes \ee_{(\alpha',\beta');J} \mapsto e^{-\left(f_{(\alpha,\beta);I}(v) + f_{(\alpha',\beta');J}(v)\right)}\ee_{(\alpha+\alpha',\beta+\beta');I+J}.
\end{equation}\par
On the other hand, we consider a composition $V_I\cdot W_J$ in $Mo_{\cE_c}(P)$ of two generators such that $V_I \neq P$ and $W_J\neq P$. We have such a composition only when
\begin{itemize}
\item $V_I$ is of type $b=0$ and $W_J$ is of type $b=1$ or
\item $V_I$ is of type $b=1$ and $W_J$ is of type $b=0$.
\end{itemize}
In both cases, the result $V_I\cdot W_J$ is generated by a generator $Z_{I+J}$ of type $b = 1$ with index $I + J$ since indices are preserved by the composition. This implies that $Z_{I+J}$ belongs to $E_2$ (resp. $E_4$) if the generator, $V_I$ or $W_J$, of type $b = 1$ belongs to $E_2$ (resp. $E_4$). The gradient trajectories starting from a point in $V_{I}$ of type $b=0$ run horizontally. Also, the gradient trajectories starting from a point in $V_{I}$ of type $b=1$ and ending at a point in $Z_K$ of type $b=1$ run horizontally. These imply that there exists the unique gradient tree $\gamma$ starting from $v\in V_I$ and $w\in W_J$ and ending at $z\in Z_{J+K}$. Since the Lagrangian $L(a,b)$ is locally the graph of $df_{(a,b);I}$ and $f_{(a,b);I}=0$ on $V_{(a,b);I}$, we have
\begin{equation}
\begin{split}\label{m-comp}
V_{(\alpha,\beta);I}\otimes V_{(\alpha',\beta');J} &\mapsto e^{-A(\gamma)}V_{(\alpha+\alpha',\beta+\beta');I+J},\\
A(\gamma) &:= f_{(\alpha,\beta);I}(z) + f_{(\alpha',\beta');J}(z).
\end{split}
\end{equation}
Since the point $z$ in (\ref{m-comp}) coincides with the point $v$ in (\ref{prod-e}), the correspondence $\iota:V_{(a,b);I}\mapsto \ee_{(a,b);I}$ is compatible with the composition of morphisms:
\begin{equation*}
\fm_2 \circ (\iota\otimes\iota)\left(V_{(\alpha,\beta);I}\otimes V_{(\alpha',\beta');J}\right) =  \iota \circ \fm_2 \left(V_{(\alpha,\beta);I}\otimes V_{(\alpha',\beta');J}\right).
\end{equation*}
\end{proof}

To summarize, the correspondence $\iota:Mo_{\cE_c}(P)\to\cV'_{\cE_c}$ is a quasi-isomorphism between DG-categories
\begin{equation*}
Mo_{\cE_c}(P)\simeq \cV'_{\cE_c}.
\end{equation*}
Also, the DG-category $DG_{\cE_c}(\Fk)$ is quasi-isomorphic to $\cV_\cE'$ by subsection \ref{DGFk}. Thus, we obtain the following.
\begin{cor}
We have a DG-equivalence
\begin{equation*}
Mo_{\cE_c}(P) \simeq DG_{\cE_c}(\Fk).
\end{equation*}
\end{cor}
This completes the proof of the main theorem (Theorem \ref{main}). 

\begin{rem}
If $k=1$ then these results coincide with Futaki-Kajiura's results\cite{fut-kaj2}.
\end{rem}

\subsection{Minimality of $Mo(P)$}\label{minimality}
We see that the full subcategory $Mo_{\cE_c}(P)\subset Mo(P)$ is minimal in the previous subsection. In the case of $\F_1$, Futaki-Kajiura's result implies that $Mo(P)$ is minimal. Because the intersection $V_{(a,b);(i_1,i_2)}=\pi(L(0,0)\cap L(a,b))$ is connected for any $(a,b),\ (i_1,i_2)\in\Z^2$ and there is a no gradient trajectory starting at $V_{(a,b);(i_1,i_2)}$ and ending at itself. However, in general, $Mo(P)$ is not minimal.
Indeed, in the case of $\Fk\ (k\geq2)$, $Mo(P)$ is not minimal. Because there exists a non-connected $V_{(a,b);(i_1,i_2)}$ and a gradient trajectory starting at some connected component and ending at another connected component.
 In this subsection, we give an example of the space of morphisms, which has non-trivial differential.\par
For $k=2$ and $(a,b)=(-7,3)$, we consider the space $Mo(P)(L(0,0),L(-7,3))$. The intersection of $L(0,0)$ with $L(-7,3)$ is expressed as
\begin{equation*}
i_1 = -7\frac{s}{1+s} + 6\frac{s^2}{1+s^2+t},\ \ \ 
i_2 = 3\frac{t}{1+s^2+t}.
\end{equation*}
If $i_2=0$ and $s<\infty$ then $t=0$. If $i_1=0$ then $s=0$ or $3\pm\sqrt{2}$. Since $(x^1,x^2)=(\frac{2s}{1+s}+2\frac{2s^2}{1+s^2+t},\frac{2t}{1+s^2+t})$, we obtain the components
\begin{equation*}
V_{(-7,3);(0,0)}  = \left\{\left(0,0\right)\right\}\cup\left\{\left(\frac{100-10\sqrt{2}}{21},0\right)\right\}\cup\left\{\left(\frac{100+10\sqrt{2}}{21},0\right)\right\}.
\end{equation*}
It is not connected. We set $V_{(-7,3);(0,0)}^0:=\left\{\left(\frac{100-10\sqrt{2}}{21},0\right)\right\}$ and $V_{(-7,3);(0,0)}^1:=\left\{\left(\frac{100+10\sqrt{2}}{21},0\right)\right\}$. For $(i_1,i_2)=(0,0)$, the gradient vector field is given by
\begin{equation*}
2\pi\left(-7\frac{s}{1+s} + 6\frac{s^2}{1+s^2+t}\right)\frac{\del}{\del x^1} + 2\pi\left(3\frac{t}{1+s^2+t} \right)\frac{\del}{\del x^2}.
\end{equation*}

\begin{figure}[h]
\center
\begin{tikzpicture}
\draw[dashed](0,0)-- (12,0)--  (4,4)--(0,4)-- cycle;
\draw(0,0) -- (12,0);
\fill[black](0,0)circle(0.06) node[below ]{$(0,0)$};
\fill[black](8.2,0)circle(0.06) node[below ]{$V_{(-7,3);(0,0)}^0$};
\fill[black](10.8,0)circle(0.06) node[below ]{$V_{(-7,3);(0,0)}^1$};
\draw [arrows = {-Stealth[scale=1.5]}] (8.2,0) -- (9.5,0);
\draw [arrows = {-Stealth[scale=1.5]}] (8.2,0) -- (4,0);
\draw [arrows = {-Stealth[scale=1.5]}] (12,0) -- (11.5,0);
\end{tikzpicture}
\caption{The gradient vector field of $f_{(-7,3);(0,0)}$.}
\label{poly2}
\end{figure}

A point $(0,0)$ does not satisfy Condition (M2) in subsection \ref{Mo(P)} and can not be a generator. For a point $\left(\frac{100-10\sqrt{2}}{21},0\right)$, the stable manifold is $\left\{\left(\frac{100-10\sqrt{2}}{21},0\right)\right\}$, so $V_{(-7,3);(0,0)}^0$ turns out to be a generator of degree zero. For a point $\left(\frac{100+10\sqrt{2}}{21},0\right)$, the stable manifold $S$ turns out to be $\left\{\left(x^1,0\right)\ \middle|\ \frac{100-10\sqrt{2}}{21}<x^1<6\right\}$. The point $\left(\frac{100+10\sqrt{2}}{21},0\right)$ is an interior point of $S\cap P\subset S$, so $V_{(-7,3);(0,0)}^1$ turns out to be a generator of degree one. For these generators, there exists a gradient tree starting at $V_{(-7,3);(0,0)}^0$ and ending at $V_{(-7,3);(0,0)}^1$. Therefore, we have
\begin{equation*}
\begin{split}
\fm_1(V_{(-7,3);(0,0)}^0) &= e^{-A} V_{(-7,3);(0,0)}^1,\\
A &= 3\pi \log\left(\frac{12+6\sqrt{2}}{12-6\sqrt{2}}\right) - 7\pi \log\left(\frac{4+\sqrt{2}}{4-\sqrt{2}}\right),
\end{split}
\end{equation*}
where $A$ is symplectic area. Thus, the space $Mo(P)(L(0,0),L(-7,3))$ has non-trivial differential, which implies that $Mo(P)$ cannot be minimal.


\begin{thebibliography}{10}
\bibitem{A}
M. Abouzaid. 
\newblock {\em Morse homology, tropical geometry, and homological mirror symmetry for toric varieties.}
\newblock {Selecta Mathematica.} 15 (2009), 189-270.

\bibitem{AKO}
D. Auroux, L. Katzarkov and D. Orlov.
\newblock {\em Mirror symmetry for Del Pezzo surfaces: Vanishing cycles and coherent sheaves.}
\newblock Invent. math. 166, (2006), 537-582.

\bibitem{BK}
A.I. Bondal and M.M. Kapranov.
\newblock {\em Enhanced triangulated categories.}
\newblock{Math. USSR-Sb.}, 1991, 70:93-107.

\bibitem{Chan}
K. Chan. 
\newblock {\em Holomorphic line bundles on projective toric manifolds from Lagrangian sections of their mirrors by SYZ transformations.}
\newblock {International Mathematics Research Notices.} 2009.24 (2009), 4686-4708.

\bibitem{coxtor}
D. Cox, J. Little, and H. Schenck.
\newblock {\em Toric varieties}, volume 124 of {\em Graduate Studies in
  Mathematics}.
\newblock American Mathematical Society, Providence, RI, 2011.

\bibitem{EL}
A. Elagin and V. Lunts.
\newblock {\em On full exceptional collections of line bundles on del Pezzo surfaces.}
\newblock {Moscow Mathematical Journal}, 16:4 , 691-709, 2016

\bibitem{FO}
K.~Fukaya and Y.-G.~Oh.
\newblock {\em Zero-loop open strings in the cotangent bundle and morse homotopy.}
\newblock {Asian J. Math.}, 1:96--180, 1997.

\bibitem{fulton93toric}
W.~Fulton.
\newblock {\em Introduction to toric varieties}.
\newblock Number 131. Princeton University Press, 1993.

\bibitem{fut-kaj1}
M.~Futaki and H.~Kajiura,
\newblock {\em Homological mirror symmetry of $\mathbb{C}P^n$ 
and their products via Morse homotopy.}
\newblock {Journal of Mathematical Physics}, 62:3, 032307, 2021. 

\bibitem{fut-kaj2}
M.~Futaki and H.~Kajiura,
\newblock {\em Homological mirror symmetry of $\F_1$ via Morse homotopy.}
\newblock {Advances in Theoretical and Mathematical Physics}, 26(8), 2611-2637, 2022.


\bibitem{hille-perling11}
L.~Hille and M.~Perling.
\newblock {\em Exceptional sequences of invertible sheaves on rational surfaces.}
\newblock {Compositio Mathematica}, 147(4):1230--1280, 2011.


\bibitem{kon94}
M.~Kontsevich.
\newblock {\em Homological algebra of mirror symmetry.}
\newblock Proceedings of the International Congress of Mathematicians, 
Vol.\ 1, 2 (Z\"urich, 1994), 120--139, Birkh\"auser, Basel, 1995. 

\bibitem{KoSo:torus}
M.~Kontsevich and Y.~Soibelman.
\newblock {\em Homological mirror symmetry and torus fibrations.}
\newblock {Symplectic geometry and mirror symmetry (Seoul, 2000)}, pages
  203--263. World Sci. Publishing, River Edge, NJ, 2001.
  
\bibitem{leung05}
N.C.~Leung. 
\newblock {\em Mirror symmetry without corrections.}
\newblock { Communications in Analysis and Geometry}, 13(2):287--331, 2005.

\bibitem{LYZ}
N.C.~Leung, S.-T.~Yau, and E.~Zaslow.
\newblock {\em From special {L}agrangian to hermitian-{Y}ang-{M}ills via
  {F}ourier-{M}ukai transform.}
\newblock  Adv.\ Theor.\ Math.\ Phys.\ {\bf 4} (2000), no.\ 6, 1319--1341.

\bibitem{HN}
H.~Nakanishi. 
\newblock{\em Homological mirror symmetry of toric Fano surfaces via Morse homotopy.}
\newblock {Journal of Mathematical Physics}, 65:5, 053501, 2024. 

\bibitem{seidel}
P. Seidel.
\newblock {\em Fukaya categories and Picard-Lefschetz theory.}
\newblock Vol. 10. European Mathematical Society, 2008.
  
\bibitem{SYZ}
A.~Strominger, S.-T.~Yau, and E.~Zaslow.
\newblock {\em Mirror symmetry is {T}-duality.}
\newblock {Nucl. Phys. B}, 479:243--259, 1996.

\bibitem{U}
K. Ueda.
\newblock {\em Homological mirror symmetry for toric del Pezzo surfaces.}
\newblock {Communications in mathematical physics}, 264(1), 71-85. 2006.

\end{thebibliography}
\end{document}